\title{Some Results on Superpatterns for Preferential Arrangements}
\author{Yonah Biers-Ariel, Rutgers University \and Yiguang Zhang, The Johns Hopkins University\and Anant Godbole, East Tennessee State University}
\date{}

\documentclass[12pt]{article}
\usepackage{amsmath,amsfonts}
\usepackage{amsthm}
\usepackage{color}
\usepackage{verbatim}

\newtheorem{theorem}{\\Theorem}

\newtheorem{proposition}{\\Proposition}
\newtheorem{definition}{\\Definition}
\newtheorem{open question}{\\Open Question}
\newtheorem{corollary}{\\Corollary}[theorem]

\begin{document}

\maketitle
\begin{abstract}
A {\it superpattern} is a  string of characters of  length $n$  that contains as a subsequence, and in a sense that depends on the context, all the smaller strings of length $k$ in a certain class.  We prove structural and probabilistic results on superpatterns for {\em preferential arrangements}, including (i) a theorem that demonstrates that a string is a superpattern for all preferential arrangements if and only if it is a superpattern for all permutations; and (ii) a result that is reminiscent of a still unresolved conjecture of Alon on the smallest permutation on $[n]$ that contains all $k$-permutations with high probability.
\end{abstract}

\section{Introduction and Statement of Results}  A {\it superpattern} is a  string of characters of  length $n$  that contains as a subsequence, and in a sense that depends on the context, all the smaller strings of length $k$ in a certain class. Specifically, given a set $X$ and a class $\mathcal{R}$ such that each object in $\mathcal{R}$ is a string of $k$ elements in $X$, a \emph{superpattern} is a string that contains all $p \in \mathcal{R}$ as subsequences. 
For example, with $X=\{1,2\}$ and $\mathcal{R}=\{11,12,21,22\}$, 
$$1221$$ is a superpattern. 

In this paper, we present some results on superpatterns for {\it preferential arrangements}, or word-patterns.
Key references in this area are \cite{burstein}, \cite {evdo}, and \cite{kont}.  Preferential arrangements (p.a.'s) of length $k$ over $X=[d]:=\{1,2,\ldots,d\}$ are $k$-strings with entries from $[d]$, for which order isomorphic representations are considered to be equivalent.  For example if $k=3, d=2$, there are seven preferential arrangements, {\it viz.} 111, 112, 121, 211, 122, 212, and 221.  If $k=d=3$, the thirteen preferential arrangements (enumerated whenever $d=k$ by the ordered Bell numbers) are 112, 121, 211, 122, 212, 221, 111, and the six permutations 123, 132, 213, 231, 312, and 321.  Note that, for example, the strings 112, 113, and 223 are order isomorphic, so above we just list the preferential arrangement 112, expressed in the traditional lexicographically minimal fashion, also known as a dense ranking system. If $k=3, d\ge 4$, there are still only 13 p.a.'s, since, for example, with $k=3$ and $d=4$, the six strings 112, 113, 114, 223, 224, 334 are each equivalent to the p.a. 112.  

A superpattern for preferential arrangements of length $k$ over $[d]$ is an $n$-long string over the alphabet $[d]$,  that contains, as a subsequence,  each of the preferential arrangements of length $k$ over $[d]$ in any one of its order isomorphic forms.  For example, a string such as 3213213 is a superpattern for $k=d=3$ {\it or with} $k=3; d=4$, and 1231241 is a superpattern with, e.g., $k=3; d=5$ or $k=3; d=4$.  Let $n(k,d)$ be the length of the shortest superpattern for all p.a.'s of length $k$ over $[d]$.

Now, let us define what we consider to be another natural object:  Let $\nu(k,d)$ be the length of shortest superpattern for $k$-long p.a.'s {\it when each of the letters  in $[d]$ must be used at least once in the superpattern}.    For $k=3$, the examples 1231231, 1231241, and 2353134, as well as the fact  that the p.a. 111 can never occur with $n=7; d\ge 6$ show that $\nu(3,d)=7$ for $d=3,4,5$\footnote{Technically, we have just shown that $\nu(3,4) \le 7; \nu(3,5)\le 7$. A proof that these values equal 7 is not too difficult, and is omitted.}  and that $\nu(3,d)=7+(d-5)$ for $d\ge 6$.  For $k=4$ the situation is more complex:  We have $\nu(4,4)=12$, as seen in Section 2, but $\nu(4,6)\le11$, as seen via the example 43514342634.  This example also shows that the assertion in \cite{burstein} that $n(k,d)=n(k,k)$ for $d>k$ needs further qualification.
\begin{open question} For fixed small values of $k,d$, calculate $\nu(k,d)$.  Establish upper and lower bounds on $\nu(k,d)$.
\end{open question}


The fact that $n(3,3)=7$ (tacitly used above) is elementary and has been shown, e.g., in \cite{martha} (in which the rather complicated waiting time distribution for a random string on $\{1,2,3\}$ to become a 3-superpattern was also studied), and the authors of \cite{burstein} conjectured that $n(k,k)=k^2-2k+4$ for each $k$, a fact that we prove for $k=4$ in Section 2 of this paper.  In the main result of Section 2, Theorem 1, we prove that for each $k$, $n(k,k)=\rho(k,k)=\rho(k)$, where $\rho(k)$ is the shortest superpattern with entries from the alphabet $\{1,2,\ldots,k\}$ that contains all the {\it permutations} of length $k$.  This somewhat surprising result has ramifications: It was shown in \cite{rado} that the quantity $\lceil k^2-\frac{7}{3}k+\frac{19}{3}\rceil$, which is smaller than $k^2-2k+4$ for $k\ge 10$, is an upper bound for $\rho(k)$.  Since, via Theorem 1, $\rho(k)$ and $n(k,k)$ are the same, this disproves the $k^2-2k+4$ conjecture.

The first paper to make the $k^2-2k+4$ conjecture was \cite{newey}, where this conjecture was presented as one of two.  The other conjecture, which, at the present time appears to be the best candidate for the true value of $n(k,k)$, is
$$n(k)=n(k,k)=\begin{cases} k^2&\text{for}\ k=1\\
k^2-k+1&\text{for}\ 2\le k\le 3\\
k^2-2k+4&\text{for}\ 4\le k\le 7\\
\ldots\\
k^2-m\cdot k+\sum_{i=1}^mi2^{m-i}&\text{for}\ 2^m\le k\le 2^{m+1}-1.
\end{cases}$$


In the rest of the paper as well, we will focus on the case $k=d$.  In Section 3, we generalize the work of \cite{abraham} and \cite{martha}  by exhibiting tight bounds on the expected waiting time until a random string, with each letter being  independently and uniformly chosen from $[k]$, becomes a superpattern for $k$-long preferential arrangements (or permutations) over $[k]$.  The waiting time is shown to be tightly concentrated around its mean. This result recalls the Alon Conjecture from \cite{arratia}, which states that the value $n=\frac{k^2}{4}$ is the threshold for a random permutation on $[n]$ to contain each of the  $k!$ $k$-permutations in an order isomorphic form -- with high or low probability.  In our result, the parent string contains repetitions, but the net result is still that each $k$-permutation must appear in a non-isomorphic form.  As seen, e.g, in \cite{linusson}, Alon's conjecture is notoriously hard -- but perhaps an approach exemplified by Theorem 1 might yield dividends.

\section{Length and Structure of Superpatterns}

In this section we will focus on the case $k=d$, and consider $n(k)$, the length of the shortest word on the alphabet $[k]$ containing all preferential arrangements of length $k$. We again consider $\rho(k)$, the length of the shortest word on the alphabet $[k]$ containing all permutations of the elements of $[k]$. We first observe that $n(k)=\rho(k)$ for $3 \le k \le 7$, and then show that $n(k)=\rho(k)$ for each $k$.

It is shown in $\cite{burstein}$ that $n(k) \le k^2-2k+4$ for all $k$, and it is conjectured that this upper bound is actually an equality. That it is a lower bound for $\rho(k)$ when $3 \le k \le 7$ was established in $\cite{newey}$, and since $\rho(k) \le n(k)$, the equality between $n(k)$ and $\rho(k)$ holds at least through $k=7$. We begin by establishing a general lower bound for $\rho(k)$ which agrees with $k^2-2k+4$ for $k=3,4$. Even though the coefficient on the leading term in Proposition 1 is $\frac{1}{2}$, making the bound asymptotically inferior to the bound $n(k)\ge k^2-c(\epsilon)k^{1.75+\epsilon}$ from \cite{kleitman}, it suffices for small values of $k$ as we will see.
\begin{proposition}$\rho (k) \ge \frac{k^2}{2}+\frac{3k}{2}-2$ for all $k\ge 2$.
\end{proposition}
\begin{proof}We will proceed by induction. As a base case, note that $\rho (2,2) = 3 = \frac{2^2}{2}+\frac{3\cdot2}{2}-2$.

Now suppose the result holds for some $k\ge 2$, and let $\sigma$ be an arbitrary word on the alphabet $[k+1]$ with length ${(k+1)^2}/{2}+{3(k+1)}/{2}-3$. We will show that $\sigma$ does not contain all permutations of $[k+1]$. Denote the first letter of $\sigma$ by $\sigma_1$, the second letter by $\sigma_2$ and so on. Clearly, each letter in $[k+1]$ must appear somewhere in $\sigma$. Let $a$ be the last letter to appear in $\sigma$; then $a$ appears at the earliest as the $k+1^{th}$ letter of $\sigma$. We will consider two cases: when $a$ first appears as the $k+1^{th}$ letter of $\sigma$ and when $a$ first appears after the $k+1^{th}$ letter. 

In the first case, the subword $\sigma_1\sigma_2...\sigma_{k+1}$ contains all elements of $[k+1]$ exactly once, so the $a$ appearing as the $k+1^{th}$ letter of $\sigma$ cannot be a part of any permutation beginning with $\sigma_2\sigma_1a$. Since $k+1 \ge 3$, permutations of this form must exist, so $a$ appears later on in $\sigma$ as well. Thus, there are at most $\frac{(k+1)^2}{2}+\frac{3(k+1)}{2}-3-(k+2)=\frac{k^2}{2}+\frac{3k}{2}-3$ letters following the first $a$ which are not $a$. However, $\sigma$ contains all permutations of $[k+1]$ beginning with $a$, so it must contain all permutations of $[k+1]\backslash \{a\}$ following the first $a$. But, by the induction hypothesis, $\frac{k^2}{2}+\frac{3k}{2}-3$ are insufficiently many letters to contain all the permutations of $k$ letters.

In the second case, $a$ first occurs at the earliest as the $k+2^{th}$ letter of $\sigma$, so it has at most $\frac{(k+1)^2}{2}+\frac{3(k+1)}{2}-3-(k+2)=\frac{k^2}{2}+\frac{3k}{2}-3$ letters following it. As before, $\sigma$ must contain all permutations of $[k+1]\backslash a$ following the first $a$, but, again, $\frac{k^2}{2}+\frac{3k}{2}-3$ are insufficiently many letters to contain all the permutations of $k$ letters. Thus, $\sigma$ does not contain all permutations of $[k+1]$.\hfill\end{proof}

The fact that $\rho(k) = n(k)$ for $1\le k \le 7$ (the 1 and 2 cases are trivial) suggests that this equality may hold for all $k$, and, in fact it does. Proving this, however, requires two new definitions.
For Definitions 1-2 and Theorem 1, let $A_k=\{a_1,\ldots,a_k\}$ be an arbitrary subset of $\mathbb{N}$ with $|A_k|=k$ and $a_1 < a_2 < ... < a_k$.
\begin{definition}A regular occurrence of a preferential arrangement in a word on the alphabet $A_k$ is an occurrence of that arrangement such that for each letter, supposing there are $i$ letters in the p.a. that are less than that letter and $j$ copies of that letter in the p.a., the letter is represented in the word by some element of $\{a_{i+1},a_{ i+2},...,a_{ i+j} \}$. 
\end{definition}
For instance, if our alphabet is $[6]$, then a regular occurrence of 112232 is one in which the $1^s$ are represented by $1^s$ or $2^s$, the $2^s$ are represented by $3^s$, $4^s$, or $5^s$, and the 3 is represented by a 6. So, 113363 and 225565 are regular occurrences of 112232, but 113343 is not. Then, a {\it regular superpattern} of length $k$ p.a.'s on $[k]$ is defined to be a string that contains a regular occurrence of all p.a.'s.   
\begin{definition}A complete word on $A_k$ is a word on $A_k$ containing every permutation of the elements of $A_k$. So, $\rho(k)$ is the length of the shortest complete word on $[k]$.
\end{definition}
Note that this second definition comes from a body of literature including, for example, $\cite{rado}$.
Now, let $\mathcal{C}_{A_k}$ be the set of complete words on $A_k$, let $\mathcal{S}_{A_k}$ be the set of superpatterns of length $k$ preferential arrangements on $A_k$, and let $\mathcal{R}_{A_k}$ be the set of regular superpatterns of length $k$ preferential arrangements on $A_k$. 
\begin{theorem}For all $A_k$ with $k \ge 2$, $\mathcal{C}_{A_k}=\mathcal{S}_{A_k}=\mathcal{R}_{A_k}$.
\end{theorem}\begin{proof}
It is clear that $\mathcal{R}_{A_k} \subseteq \mathcal{S}_{A_k} \subseteq \mathcal{C}_{A_k}$, so it remains to show that $\mathcal{C}_{A_k} \subseteq \mathcal{R}_{A_k}$. We proceed by induction. As a base case, note that for $A_2$ any $\sigma_\mathcal{C} \in \mathcal{C}_{A_2}$ contains either the subsequence $a_1a_2a_1$ or $a_2a_1a_2$, so $\sigma_\mathcal{C} \in \mathcal{R}_{A_2}$. Now suppose that $\mathcal{C}_{A_{k-1}} \subseteq \mathcal{R}_{A_{k-1}}$. Choose any $A_k$ (hereafter, we simply call this set $A$), choose some $\sigma_\mathcal{C} \in \mathcal{C}_A$, and let $\pi$ be an arbitrary preferential ordering of length $k$. Let $\pi'$ be the portion of $\pi$ following its first letter. We will now find a regular occurence of $\pi$ in $\sigma_\mathcal{C}$ in both of two cases.

$\textit{Case 1:}$
Suppose that the first letter in $\pi$ occurs just once in $\pi$. Call this first letter $c$, and let $i$ be the number of letters in $\pi$ less than $c$. Then, any regular occurrence of $\pi$ represents $c$ using $a_{i+1}$. Now, let $B= A\backslash\{a_{i+1}\} = \{b_1, b_2, ..., b_{k-1} \}$ where $b_1 < b_2 <...< b_{k-1}$. Note that $b_j = a_j$ for $j< i+1$ and $b_j = a_{j+1}$ for $j \ge i+1$. Now, let $\sigma_\mathcal{C}'$ be the portion of $\sigma_\mathcal{C}$ following its first occurrence of $a_{i+1}$ with all the $a_{i+1}$'s removed. Since $\sigma_\mathcal{C} \in \mathcal{C}_A$, it follows that $\sigma_\mathcal{C}' \in \mathcal{C}_B$. By the induction hypothesis, then, $\sigma_\mathcal{C}' \in \mathcal{R}_B$, so it contains a regular occurrence of $\pi'$. We claim that appending $a_{i+1}$ to the beginning of this occurrence gives a regular occurrence of $\pi$. First, consider any letter $d < c$. Suppose there are $j$ instances of $d$ in $\pi$ and $d$ is greater than $l$ other letters in $\pi$ noting that $l+ j \le i$ must hold. Then, $d$ must be represented in our regular occurrence of $\pi'$ by some element of $\{a_{l+1},...,a_{ l+j} \}$. Since there are also $j$ instances of $d$ and $l$ letters less than $d$ in $\pi'$, we know that $d$ is represented in the regular occurrence of $\pi'$ by some element of $\{b_{l+1},...,b_{l+j} \}$, and this set is equivalent to $\{a_{l+1},...,a_{ l+j} \}$ because all the indices are less than $i+1$. Now consider $c$. For our occurrence to be regular, $c$ must be represented using $a_{i+1}$, and it is. Finally consider $d > c$. Again, suppose there are $j$ instances of $d$ in $\pi$ and $d$ is greater than $l$ other letters in $\pi$ noting that, this time, $l \ge i+1$. As before, $d$ must be represented in our regular occurrence by some element of $\{a_{l+1},...,a_{ l+j} \}$. Now, though, there are $j$ instances of $d$ and $l-1$ letters less than $d$ in $\pi'$, so $d$ is represented in the regular occurrence of $\pi'$ by some element of $\{ b_l,...,b_{l+j-1} \}$, and this set is equivalent to $\{a_{l+1},...,a_{l+j} \}$ because all indices are at least $i+1$. Thus, each letter in $\pi$ is correctly represented, and we have a regular occurrence.

$\textit{Case 2:}$
Suppose that the first letter in $\pi$ occurs $p$ times with $p>1$. Call this first letter $c$, and let $i$ be the number of letters in $\pi$ less than $c$. Then, any regular occurrence of $\pi$ represents $c$ using an element of $\{a_{i+1},..., a_{i+p} \}$. Let $a_t$ be the last of those elements to make its first appearance in $\sigma_\mathcal{C}$, and let $\sigma_\mathcal{C}'$ be the portion of $\sigma_\mathcal{C}$ following the first occurrence of $a_t$ with all subsequent $a_t$'s removed. As in case 1, let $B= A\backslash \{a_t\} = \{b_1, b_2, ..., b_{k-1} \}$ where $b_1 < b_2 <...< b_{k-1}$ and note that $b_j = a_j$ for $j< t$ and $b_j = a_{j+1}$ for $j \ge t$. Now, $\sigma_\mathcal{C}' \in \mathcal{C}_B$, and by the induction hypothesis, $\mathcal{C}_B \subseteq \mathcal{R}_B$, so $\sigma_\mathcal{C}'$ contains a regular occurrence of $\pi'$. Since there are $p-1$ occurrences of $c$ and $i$ letters less than $c$ in $\pi'$, $c$ must be represented in our regular occurrence by some element of $\{b_{i+1},...,b_{i+p-1} \}$ which is equivalent to $\{a_{i+1},...,a_{i+p} \}\backslash \{a_t\}$. Let $a_s$ be the element in this set which represents $c$, and note that it must occur before the first appearance of $a_t$ by our choice of $a_t$. We will show that appending $a_s$ to the beginning of the regular occurrence of $\pi'$ gives a regular occurrence of $\pi$. As already noted, $c$ is represented by $a_s \in \{a_{i+1},...,a_{ i+p} \}$, and for any $d<c$, the proof that $d$ is correctly represented is identical to the proof in case 1. For $d>c$, suppose there are $j$ instances of $d$ in $\pi$ and $d$ is greater than $l$ other letters in $\pi$ noting that $l \ge i+p$. Then, $d$ must be represented in our regular occurrence by some element of $\{a_{l+1},...,a_{l+j} \}$. There are $j$ instances of $d$ and $l-1$ letters less than $d$ in $\pi'$, so $d$ is represented in the regular occurrence of $\pi'$ by some element of $\{ b_l,...,b_{l+j-1} \}$, and this set is equivalent to $\{a_{l+1},...,a_{l+j} \}$ because all indices are at least $i+p \ge t$. Thus, we have found a valid regular occurrence of $\pi$ in $\sigma_\mathcal{C}$.
\hfill\end{proof}
Theorem 1 is useful in two regards. First, it allows us to apply everything known about complete words to superpatterns of preferential arrangments. As noted in the introduction, this immediately gives us that $\lceil k^2-\frac{7}{3}k+\frac{19}{3}\rceil$ is an upper bound on $n(k)$, thereby disproving a long-standing conjecture. Theorem 1 could also potentially help in finding lower bounds for $n(k)$ because proving that no words of a certain length are regular superpatterns may be easier than proving that no words are superpatterns, but this approach has not been fruitful so far.

\section{Random Superpatterns}
Finally, we will prove a result regarding random superpatterns. Consider the following random process: beginning with an empty word $\overline{W}$, at each timestep we choose a letter, uniformly at random, from the alphabet $[k]$. We then concatenate this value onto the end of $\overline{W}$ and check to see if the augmented $\overline{W}$ is a superpattern for all $k$-long preferential arrangements on $[k]$ (or, equivalently, a complete word on $[k]$). We are interested in the value of $E[X_k]$ where $X_k$ is the first timestep at which $\overline{W}$ is a superpattern on $[k]$. This problem was first considered by Godbole and Liendo in \cite{martha}. There, the authors found values for $E[X_2]$ and $E[X_3]$ as well as the exact distributions of $X_2$ and $X_3$; here we will apply a previous result to give a general upper bound on $E[X_k]$, and then prove a lower bound. These bounds will be asymptotically equivalent and together prove that $E[X_k] \sim k^2 \log k$.  The distribution of $X_k$ appears to be intractable for $k\ge4$.  

Abraham et al. consider a similar problem in \cite{abraham}; they were interested in omnisequences which must contain every $k$-letter word on $[k]$, and find that $Z_k$, the expected number of randomly chosen letters necessary to produce an omnisequence is asymptotically $k^2 \log k$ (with error terms as described below). This work has connections to the coupon collector problem as studied in \cite{flajolet}, \cite{wilf},  and \cite{zeilberger}; these connections carry forward to the work in this section.  Since every omnisequence of $k$-letter words on $[k]$ is also a superpattern on $[k]$, we obtain the following corollary of Abraham's work, where $\gamma\approx .577$ is Euler's constant and $\log$ denotes the natural logarithm.
\begin{theorem} $E[X_k]\le k^2(\log(k)+\gamma+O(k^{-1}))$ for all $k$.
\end{theorem}
The next theorem provides a similar lower bound.
\begin{theorem} For all $k$, $E[X_k] \ge k^2(\log(k)+\gamma-1+O(k^{-1})).$
\end{theorem}
\begin{proof}
Fix $k$, and let $\overline{W}$ be a superpattern on $[k]$. We will define a word $W$ which $\overline{W}$ must contain, and then we will calculate $E[Y_k]$, where $Y_k$ is the number of letters used before $W$ appears. Let $k_1$ be the last element of $[k]$ to make its first appearance in $\overline{W}$ and let this appearance be the $p_1\textsuperscript{th}$ letter of $\overline{W}$. Then let $k_2$ be the last element of $[k]\backslash\{k_1\}$ to make its first appearance after $\overline{w}_{p_1}$, and let this appearance be the $p_2\textsuperscript{th}$ letter of $\overline{W}$. In general, let $k_i$ be the last element of $[k]\backslash\{k_1, k_2,...,k_{i-1}\}$ to make its first appearance after $\overline{w}_{p_{i-1}}$, and let this appearance occur at the $p_i\textsuperscript{th}$ letter of $\overline{W}$. Now, $W$ consists of $k$ blocks; the first block contains all letters in $[k]$, the second block contains all letters in $[k]\backslash\{k_1\}$, and so on. Because $\overline{W}$ is a superpattern on $[k]$, in particular because it contains $k_1k_2...k_k$ as a subsequence, it must contain $W$ as a subsequence.  Note that the string $k_1k_2...k_k$ is not necessarily the last permutation to occur; e.g., the superpattern 1231213 has $k_1k_2k_3=321$ even though the last permutation to appear is 213.

Now, let $Y_{k,i}$ be the number of timesteps needed to form the $i\textsuperscript{th}$ block of $W$. This block must contain $k-i+1$ distinct letters from the set $[k]\backslash\{k_1,k_2,...,k_{i-1}\}$. At each timestep, we add one of $k$ possible letters; there are $k-i+1$ possibilities for the first distinct letter, and so it appears after $T_{i,1}$ timesteps where $T_{i,1}$ follows a geometric distribution with parameter $\frac{k-i+1}{k}$. Then, there are $k-i$ possibilities for the second distinct letter and so on. Therefore, to find the $j\textsuperscript{th}$ distinct letter requires waiting $T_{i,j}$ timesteps, where $T_{i,j}$ follows a geometric distribution with parameter $\frac{k-i-j+2}{k}$. Thus, we have that
$$E[Y_{k,i}]=\sum_{j=1}^{k-i+1} E[T_{i,j}] =\sum_{j=1}^{k-i+1} \frac{k}{k-i-j+2}=\sum_{j=i}^{k} \frac{k}{k-j+1}.$$
Using the fact that $E[Y_k] = \sum_{i=1}^k E[Y_{k,i}]$, we now get
\begin{align*}E[Y_k]&=\sum_{i=1}^k \sum_{j=i}^{k} \frac{k}{k-j+1}= \sum_{j=1}^k \sum_{i=1}^j \frac{k}{k-j+1} =\sum_{j=1}^k j \frac{k}{k-j+1}=k\sum_{j=1}^k \frac{j}{k-j+1}.
\end{align*}
Lastly, we take $j \mapsto k-j+1$ to see that
\begin{align*}
k\sum_{j=1}^k \frac{j}{k-j+1}&=k\sum_{j=1}^k \frac{k-j+1}{j}\\&=k^2\sum_{j=1}^k \frac{1}{j} -k\sum_{j=1}^k 1 +k \sum_{j=1}^k \frac{1}{j} 
\\
&\ge k^2 (\log (k)+\gamma+O(k^{-1})) -k^2\\ 
&= k^2(\log(k) + \gamma -1 +O(\log k/k)).
\end{align*}
Therefore, $E[X_k] \ge E[Y_k] \ge k^2(\log(k) + \gamma -1 +O(k^{-1})))$.
\end{proof}
\begin{corollary} As $k \rightarrow \infty$, $E[X_k] \sim k^2 \log (k)$.
\end{corollary}
We are also interested in the concentration of $X_k$ about its mean; in particular we would like to find a lower bound which $X_k$ exceeds with high probability and an upper bound which $X_k$ falls below with high probability. A conjecture of Noga Alon states that for a random permutation of $[n]$ to contain all $k!$ permutations of $[k]$ with high probability, one must have $n=\frac{k^2}{4}(1+o(1))$ \cite{arratia}. While this conjecture has remained open for fifteen years, we will prove an analogue regarding superpatterns of superpatterns of preferential arrangements (equivalently superpatterns of permutations) when restricted to the alphabet $[k]$. As in the previous proof, we will find it useful to work with $Y_k$ instead of $X_k$, and the first step is to bound the variance of $Y_k$. 
\begin{theorem} var$(Y_k)=\Theta(k^3)$.
\end{theorem}
\begin{proof}
The proof closely follows the proof of the previous theorem. As before, we begin by calculating $var[Y_{k,i}]$. Recall that each $Y_{k,i}$ is the sum of random variables $T_{i,j}$ each of which follows a geometric distribution with parameter $\frac{k-i-j+2}{k}.$ Note that the $Y_{k,i}$ are all mutually independent, as are all the $T_{i,j}$. Therefore, we have that
\[var[Y_{k,i}]=\sum_{j=1}^{k-i+1} var[T_{i,j}] =\sum_{j=1}^{k-i+1} \frac{i+j-2}{k} \cdot \frac{k^2}{(k-i-j+2)^2}=k\sum_{j=1}^{k-i+1} \frac{i+j-2}{(k-i-j+2)^2}.\]
Making the substitution $j \mapsto j+i-1$, this becomes
\[k \sum_{j=i}^{k} \frac{j-1}{(k-j+1)^2}.\]
Since the $Y_{k,i}\textsuperscript{s}$ are independent, we have $var[Y_k] = \sum_{i=1}^k var[Y_{k,i}]$, and so we now get 
\begin{align*}
var[Y_k] &= \sum_{i=1}^k k\sum_{j=i}^k \frac{j-1}{(k-j+1)^2}
\\
& = k\sum_{j=1}^k \sum_{i=1}^j \frac{j-1}{(k-j+1)^2}
\\
& = k \sum_{j=1}^k \frac{j(j-1)}{(k-j+1)^2}.
\end{align*}
Make the substitution $j \mapsto k-j+1$ to get
\begin{align*}
var[Y_k] & = k\sum_{j=1}^k \frac{(k-j+1)(k-j)}{j^2}
\\
& =k \sum_{j=1}^k \frac{k^2-2kj+j^2+k-j}{j^2}
\\
& = (k^3+k^2)\sum_{j=1}^k \frac{1}{j^2} -(2k^2 +k)\sum_{j=1}^k \frac{1}{j} +k^2
\\
& \sim \frac{\pi^2}{6} k^3.
\end{align*}
\end{proof}
The corresponding result for $Z_k$, i.e. that $var[Z_k] = \Theta(k^3)$ is also proved by Abraham et al. in \cite{abraham}. Now that we have a handle on the variances of $Y_k$ and $Z_k$, we use Chebyshev's inequality to find bounds between which $X_k$ falls with high probability. 
\begin{theorem} With high probability, we have that $k^2\log k - (1 -\gamma) k^2 - \omega(1) k^{\frac{3}{2}} \le X_k \le  k^2\log k + \gamma k^2+\omega(1) k^{\frac{3}{2}}$ where $\omega(1)$ is any sequence tending to $\infty$.
\end{theorem}
\begin{proof}
We begin by showing that $P[Y_k > k^2\log k - (1-\gamma) k^2 - \omega(1) k^{\frac{3}{2}}] \rightarrow 1$. It holds that
\begin{align*}
P[Y_k \le k^2\log k - (1-\gamma) k^2 - \omega(1) k^{\frac{3}{2}}] &\le P[|Y_k-E[Y_k]| \ge \omega(1) k^\frac{3}{2}]
\\
& \le \frac{\Theta(k^3)}{(\omega(1) k^\frac{3}2)^2} 
\\
&\to0
\end{align*}
Next, we use a similar argument to show that $P[Z_k < k^2 \log k + \gamma k^2 + \omega(1) k^\frac{3}{2}]$ with high probability.
\begin{align*}
P[Z_k \ge k^2 \log k + \gamma k^2 +\omega(1)k^\frac{3}{2}] & \le P[|Z_k-E[Z_k]| \ge \omega(1) k^\frac{3}{2}]
\to0.
\end{align*}
Since $Y_k \le X_k \le Z_k$, these two inequalities suffice to show that $k^2\log k - (1-\gamma) k^2 - \omega(1) k^{\frac{3}{2}} \le X_k \le  k^2\log k + \gamma k^2+\omega(1) k^{\frac{3}{2}}$ with high probability.
\end{proof}
Therefore, $X_k$ lies, with high probability, in an interval of length $k^2$ around its expected value.  It would be interesting to be able to nail down better asymptotic estimates in the above argument.  What are $E(X_k)$ and $var (X_k)$?
%

\section{Acknowledgments}  The research of the first and third authors was supported by NSF Grant 1004624.  The research of the second author was supported by the Acheson J. Duncan Fund for the Advancement of Research in Statistics.

\end{document}